\documentclass[11pt]{amsart}
\usepackage{amsfonts,amssymb,amsthm}
\usepackage{amsmath,amscd}
\usepackage{pstricks}
\usepackage{pstricks,pst-node}
\usepackage{mathrsfs}
\usepackage[all]{xy}

\def\bffke{\mathtt{e}}
\def\bffkf{\mathtt{f}}

\def\bffkk{{\mathtt 1}}
\newcommand{\ttk}{\mathtt{1}}

\newcommand{\bssg}{\boldsymbol{\sg}}

\newcommand{\bse}{\boldsymbol{e}}

\newcommand{\tf}{\ti{f}}
\DeclareMathAlphabet{\mathpzc}{OT1}{pzc}{m}{it}

\theoremstyle{plain}
\newtheorem{Thm}{Theorem}[section]
\newtheorem{Prop}[Thm]{Proposition}
\newtheorem{Lem}[Thm]{Lemma}
\newtheorem{Coro}[Thm]{Corollary}
\theoremstyle{definition}

\newtheorem{Def}[Thm]{Definition}

\numberwithin{equation}{section}

\newcommand{\bfj}{{\mathbf{j}}}

\newcommand{\bfl}{{\mathbf{0}}}

\newcommand{\bfU}{{\mathbf{U}}}

\newcommand{\bfS}{{\mathbf{S}}}

\newcommand{\bfTr}{{\mathbf T}_\vtg(n,r)}

\def\sfz{{\mathsf z}}

\def\sA{{\mathcal A}}

\def\sL{{\mathcal L}}

\def\sN{{\mathcal N}}

\def\sU{{\mathcal U}}

\def\sX{{\mathcal X}}

\def\sZ{{\mathcal Z}}

\newcommand{\mbn}{\mathbb N}
\newcommand{\mbq}{\mathbb Q}

\newcommand{\mbz}{\mathbb Z}

\newcommand{\ttx}{\mathtt{x}}
\newcommand{\tty}{\mathtt{y}}

\newcommand{\tte}{\mathtt{e}}

\newcommand{\ttf}{\mathtt{f}}

\newcommand{\ttm}{\mathtt{m}}

\newcommand{\spann}{\operatorname{span}}
\newcommand{\diag}{\operatorname{diag}}

\def\ro{\text{\rm ro}}
\def\co{\text{\rm co}}

\newcommand{\la}{{\lambda}}
\newcommand{\La}{\Lambda}

\newcommand{\Ga}{{\Gamma}}

\newcommand{\dt}{\delta}

\newcommand{\up}{v}

\newcommand{\vep}{\varepsilon}

\newcommand{\bt}{\beta}
\newcommand{\sg}{\sigma}

\def\bpa#1#2{\left({#1\atop #2}\right)}

\def\ggp#1#2{\left[\kern-3.2pt\left[{#1\atop #2}\right]\kern-3.2pt\right]}

\def\leq{\leqslant}\def\geq{\geqslant}
\def\le{\leqslant}

\newcommand{\ot}{\otimes}

\newcommand{\han}{\subseteq}

\newcommand{\h}{\widehat}
\newcommand{\ti}{\widetilde}

\newcommand{\tA}{{}^t\!A}

\newcommand{\lra}{\longrightarrow}
\newcommand{\ra}{\rightarrow}

\newcommand{\vtg}{{\!\vartriangle}}

\newcommand{\afSrmbq}{{\mathcal S}_{\vtg}(n,r)_{\mathbb Q}}

\newcommand{\afgl}{\widehat{\frak{gl}}_n}

\newcommand{\afal}{{\alpha}^\vartriangle}

\newcommand{\afbse}{\boldsymbol e^\vartriangle}

\newcommand{\afE}{E^\vartriangle}
\newcommand{\afSr}{{\mathcal S}_{\vtg}(n,r)_\sZ}
\newcommand{\afSrl}{{\mathcal S}_{\vtg}(n,r)}

\newcommand{\afbfSr}{{\boldsymbol{\mathcal S}}_\vtg(n,r)}
\newcommand{\afbfSrr}{{\boldsymbol{\mathcal S}}_\vtg(r,r)}

\newcommand{\afmbnn}{\mathbb N_\vtg^{n}}
\newcommand{\afmbzn}{\mathbb Z_\vtg^{n}}

\newcommand{\afThn}{\Theta_\vtg(n)}

\newcommand{\afThnpm}{\Theta_\vtg^\pm(n)}
\newcommand{\afThnp}{\Theta_\vtg^+(n)}
\newcommand{\afThnm}{\Theta_\vtg^-(n)}

\newcommand{\afThnr}{\Theta_\vtg(n,r)}

\newcommand{\afMnc}{M_{\vtg,n}(\mathbb Q)}

\newcommand{\afLa}{\Lambda_\vtg}
\newcommand{\afLanr}{\Lambda_\vtg(n,r)}

\begin{document}
\title{Presenting affine Schur Algebras}

\author{Qiang Fu}
\address{Department of Mathematics, Tongji University, Shanghai, 200092, China.}
\email{q.fu@hotmail.com, q.fu@tongji.edu.cn}

\author{Mingqiang Liu}
\address{Department of Mathematics, Tongji University, Shanghai, 200092, China.}
\email{mingqiangliu@163.com}


\thanks{Supported by the National Natural Science Foundation
of China}

\begin{abstract}
The universal enveloping algebra $\sU(\afgl)$ of $\afgl$ was realized in \cite[Ch. 6]{DDF} using affine Schur algebras. In particular some explicit multiplication formulas in affine Schur algebras were derived. We use these formulas to study the structure of affine Schur algebras. In particular, we give a presentation of the affine Schur algebra $\afSrmbq$ over $\mbq$.
\end{abstract}


 \sloppy \maketitle
\section{Introduction}
 Beilinson, Lusztig, and
MacPherson (BLM) gave a geometric realization for  the quantum enveloping algebra $\bfU(\frak{gl}_n)$ of $\frak{gl}_n$ over $\mbq(\up)$ via
$q$-Schur algebras in \cite{BLM}.
The remarkable BLM's work has many applications. Using BLM's work, it was proved  in \cite{Du95} that the natural algebra homomorphism from the Lusztig integral form of quantum $\frak{gl}_n$ to the $q$-Schur algebra over $\sZ$ is surjective, where $\sZ=\mbz[\up,\up^{-1}]$. The integral form of quantum $\frak{gl}_n$ was realized in \cite{Fu15a} and the Frobenius--Lusztig kernel of type $A$ was realized in \cite{Fu15b}. Furthermore, BLM's work can be used to investigate the presentation of $q$-Schur algebras (cf. \cite{DG,DP03}).

The affine quantum Schur algebra is the affine version of the $q$-Schur algebra and it has several equivalent definitions (see \cite{GV,Gr99,Lu99}).
Let $\bfU(\afgl)$ be the quantum enveloping algebra of the loop algebra of $\mathfrak {gl}_n$. A conjecture about the realization of $\bfU(\afgl)$ using affine quantum Schur algebras was formulated in \cite[5.5(2)]{DF10}. This conjecture has been proved in the classical ($v=1$) case in \cite[Ch. 6]{DDF}, and in the quantum case in \cite{DF13}. These results have important application to the investigation of the integral affine quantum Schur--Weyl reciprocity (cf. \cite{Fu13,Fu12,DF14}).

The presentation of affine quantum Schur algebras is useful in the
investigation of categorifications of the affine quantum Schur algebras
(cf. \cite[0.1]{MT1} and \cite{MT2}).
The presentation of the affine quantum Schur algebra $\afbfSr$ over $\mbq(\up)$ is given
in \cite{DGr,Mc07} under the assumption that
$n>r$. The presentation of the affine quantum Schur algebra $\afbfSrr$ is given  in \cite[Thm. 5.3.5]{DDF}. However the presentation problem of $\afbfSr$ is much more complicated in the $n<r$ case (cf. \cite[Rem. 5.3.6]{DDF}).
In this paper, we give a presentation of the classical affine Schur algebra $\afSrmbq$ over $\mbq$ for any $n,r$.

We organize this paper as follows. We recall the definition of the algebra $\sU(\afgl)$ and the affine Schur algebra $\afSrmbq$ in \S2.
In \S3, we recall the multiplication formulas in affine Schur algebras established in \cite[Ch. 6]{DDF}. These formulas are essential in studying the presentation of affine Schur algebras. Furthermore we will construct the PBW-basis for affine Schur algebras in \ref{basis for ASA}. We will construct the algebra $\bfTr$ by generators and relations in \S4 and prove that $\bfTr$ is isomorphic to the affine Schur algebra $\afSrmbq$ in \ref{main}.

\section{The algebra $\sU(\afgl)$ and the affine Schur algebra $\afSrmbq$}

Let $\afMnc$ be the set of all $\mbz\times\mbz$ matrices $A=(a_{i,j})_{i,j\in\mbz}$ with $a_{i,j}\in\mbq$ such that
\begin{itemize}
\item[(a)] For $i,j\in\mbz$, $a_{i,j}=a_{i+n,j+n}$; \item[(b)] For every $i\in\mbz$, the set $\{j\in\mbz\mid a_{i,j}\not=0\}$ is finite.
\end{itemize}
For $i,j\in\mbz$,  denote by $\afE_{i,j}=(e^{i,j}_{k,l})_{k,l\in\mbz}$ satisfying
\begin{equation*}
e_{k,l}^{i,j}=
\begin{cases}1~~~\text{if}~k=i+sn,l=j+sn~\text{for~some} s\in\mbz;\\
0~~~\text{otherwise}.\end{cases}
\end{equation*}
Let $\afgl:=\mathfrak{gl}_n(\mathbb{Q})\otimes\mathbb{Q}[t,t^{-1}].$
Clearly, the map
$$\afMnc\lra\afgl,\,\,\,\afE_{i,j+ln}\longmapsto E_{i,j}\ot t^l, \,1\le i,j\le n,l\in\mbz,$$
is a Lie algebra isomorphism.
We will identify the loop algebra $\afgl$ with $\afMnc$ in the sequel.

Let $\afThn:=\{A\in\afMnc\mid a_{i,j}\in \mathbb{N}\}$ and
$\afThnpm:=\{A\in\afThn\mid \text{for~any}\ i, a_{i,i}
=0\}.$
Let $\afThnp:=\{A\in\afThn\mid a_{i,j}=0~\text{for}~i\geq j\}$ and $\afThnm:=\{A\in\afThn\mid a_{i,j}=0 ~\text{for}~i\leq j\}.$
For $A\in\afThnpm$ write $A=A^++A^-$ with $A^+\in\afThnp$ and $A^-\in\afThnm$.

Let $\sU(\afgl)$ be the universal enveloping algebra of the loop algebra $\afgl$. For $t\in\mbn$ let $\big({H_i\atop
t}\big):=\frac{H_i(H_i-1)\cdots(H_i-t+1)}{t!}\in\sU(\afgl)$, where $H_i:=\afE_{i,i}$. For $\lambda\in\afmbnn$, denote $$\bigg({H\atop\lambda}\bigg)=\prod\limits_{1\leq i\leq n}\bigg({H_i\atop
\lambda_i}\bigg).$$
Let
\begin{equation}\label{L+-}
\sL^+=\{(i,j)\mid 1\leq i\leq n,\ j\in\mbz,\,i<j\}\text{ and } \sL^-=\{(i,j)\mid 1\leq i\leq n,\ j\in\mbz,\,i>j\}
\end{equation}
Then the set
\begin{equation}\label{basis2}
\big\{\prod_{(i,j)\in\sL^+}({{\afE_{i,j}}})^{a_{i,j}}\bigg({H\atop\lambda}\bigg)
\prod_{(i,j)\in\sL^-}({{\afE_{i,j}}})^{a_{i,j}}\,\big |\, \la\in\afmbnn,\,A\in\afThnpm\big\}
\end{equation}
forms a basis of $\sU(\afgl)$,
where the products are taken with respect to any fixed total order on $\sL^+$ and $\sL^-$.

Let $I=\mbz/n\mbz$ and we identify $I$ with $\{1,2,\cdots,n\}$.
Let $C=(c_{i,j})_{i,j\in I}$ be the Cartan matrix of affine type $A_{n-1}$. For $s\in\mbz$ with $s\not=0$  let
$Z_s=\sum_{1\leq h\leq n}\afE_{h,h+sn}.$
According to \cite[Thm. 6.1.1]{DDF} we have the following result.
\begin{Lem}
The algebra $\sU(\afgl)$ is the $\mbq$-algebra generated by $E_i=\afE_{i,i+1}$, $F_i=\afE_{i+1,i}$, $H_i$, $Z_s$, $i\in I$, $s\in\mbz$, $s\not=0$ with relations ($i,j\in I,\,s,t\in\mbz,\,s,t\not=0$)
\begin{itemize}
\item[(UR1)] $H_iH_j=H_jH_i$;
\item[(UR2)]
$H_iE_{j}-E_{j}H_i=(\dt_{i,j}-\dt_{i,j+1})E_{j},\quad H_iF_{j}-F_{j}H_i=(\dt_{i,j+1}-\dt_{i,j})F_{j};$
\item[(UR3)]
$E_{i}F_{j}-F_{j} E_{i} = \delta_{ij}(H_j-H_{j+1})$
\item[(UR4)]
$\displaystyle\sum_{a+b=1-c_{i,j}}(-1)^a\bigg({1-c_{i,j}\atop a}\bigg)
E_i^{a}E_jE_i^{b}=0$ for $i\not=j$;
\item[(UR5)]
$\displaystyle\sum_{a+b=1-c_{i,j}}(-1)^a\bigg({1-c_{i,j}\atop a}\bigg)
F_i^{a}F_jF_i^{b}=0$ for $i\not=j$;
\item[(UR6)]
$E_iZ_s=Z_s E_i$, $F_iZ_s=Z_s F_i$, $H_iZ_s=Z_s H_i$, $Z_sZ_t=Z_t Z_s$.
\end{itemize}
\end{Lem}

\begin{Coro}\label{lem0}
The universal enveloping algebras $\sU(\afgl)$ is the $\mathbb{Q}$-algebra generated by $E_i$, $F_i$, $H_i$, $E_{i,i+mn}$, $i\in I, m\in\mbz\backslash\{0\}$ with relations  (UR1)--(UR5) and
\begin{itemize}
\item[(UR6)$'$]
$E_{i,i+mn}H_j=H_jE_{i,i+mn};$
${}E_{i,i+mn}E_{j,j+ln}=E_{j,j+ln}E_{i,i+mn};$
\item[(UR7)$'$]
$\sum\limits_{1\leq i\leq n}E_{i,i+mn}E_j=E_j\sum\limits_{1\leq i\leq n}E_{i,i+mn};$
$\sum\limits_{1\leq i\leq n}E_{i,i+mn}F_j=F_j\sum\limits_{1\leq i\leq n}E_{i,i+mn};$
\item[(UR8)$'$]
$[X_{i,m},[[\ldots[E_i,E_{i+1}],\ldots],E_n]]
=E_{1,1+mn}-E_{i,i+mn}$ for $i\neq 1$ and $m>0$, where $X_{i,m}=[[\ldots,[[E_1,E_{2,2+(m-1)n}],E_2],\ldots],E_{i-1}]$;
\item[(UR9)$'$]
$[[[F_n,\ldots,[F_{i+1},F_{i}]\ldots]],Y_{i,m}]
=E_{1,1-mn}-E_{i,i-mn}$
for $i\neq 1$ and $m>0$, where $Y_{i,m}=[F_{i-1},[\ldots,[F_2,
[E_{2,2-(m-1)n},F_1]],\ldots]]$.
\end{itemize}
\end{Coro}
\begin{proof}
Let $\sU$ be the $\mathbb{Q}$-algebra generated by $E_i, F_i, H_i$ and $E_{i,i+mn}$ $(i\in I, m\in\mbz\backslash\{0\})$ with the given presentation. There is a surjective $\mathbb{Q}$-algebra homomorphism
$f: \sU\rightarrow \sU(\afgl)$ such that $f(E_i)=\afE_{i,i+1}$, $f(F_i)=\afE_{i+1,i}$, $f(H_i)=\afE_{i,i}$ and $f(E_{i,i+mn})=\afE_{i,i+mn}$.
On the other hand, there is a $\mathbb{Q}$-algebra homomorphism
$g: \sU(\afgl)\rightarrow \sU$ such that $g(\afE_{i,i+1})=E_i$, $g(\afE_{i+1,i})=F_i$, $g(\afE_{i,i})=H_i$ and $g(\sum\limits_{1\leq i\leq n}\afE_{i,i+mn})=\sum\limits_{1\leq i\leq n}E_{i,i+mn}$.

It is clear that we have $f\circ g=id$. In order to prove that $g\circ f=id$, it is enough to prove that $g(\afE_{i,i+mn})=E_{i,i+mn}$ for $i\in I$, $m\in\mbz\backslash\{0\}$.
We use induction on $m$.
By (UR8)$'$ we have $g(\afE_{1,1+n}-\afE_{i,i+n})=E_{1,1+n}-E_{i,i+n}$ for $1\leq i\leq n$.
This implies that $g(n\afE_{1,1+n})-g(\sum\limits_{1\leq i\leq n}\afE_{i,i+n})=nE_{1,1+n}-\sum\limits_{1\leq i\leq n}E_{i,i+n}$. It follows that $g(\afE_{1,1+n})=E_{1,1+n}$ and hence
$g(\afE_{i,i+n})=E_{i,i+n}$ for any $1\leq i\leq n$.
Assume now that $m>1$. Then by induction and (UR8)$'$, we conclude that $g(\afE_{1,1+mn}-\afE_{i,i+mn})=E_{1,1+mn}-E_{i,i+mn}$ for any $i>1$. It follows that $g(n\afE_{1,1+mn})-g(\sum\limits_{1\leq i\leq n}\afE_{i,i+mn})=nE_{1,1+mn}-\sum\limits_{1\leq i\leq n}E_{i,i+mn}$ and hence $g(\afE_{i,i-mn})=E_{i,i-mn}$ for $m>0$.
Similarly, we have $g(\afE_{i,i-mn})=E_{i,i-mn}$ with $m>0$.
The assertion follows.
\end{proof}

Let $\sZ=\mbz[\up,\up^{-1}]$, where $\up$ is an indeterminate.
For $r\geq 0$ let $\afSr$\footnote{The
algebra $\afSr$ is denoted by  $\frak U_{r,n,n;\sA}$ in \cite[1.10]{Lu99}.} be the algebra over $\sZ$ defined in \cite[1.10]{Lu99}.
Let $\afbfSr=\afSr\ot_\sZ\mbq(\up)$. The algebras $\afSr$ and $\afbfSr$ are called affine quantum Schur algebras.
The algebra $\afSr$ has a normalized  $\sZ$-basis $\{[A]\mid A\in\afThnr\}$,
where $
\afThnr=\big\{A\in\afThn\,\big|\,\sg(A):=\sum_{1\leq i\leq n,\,
j\in\mbz}a_{i,j}=r\big\}.$

Let $\afSrmbq=\afSr\otimes \mathbb{Q}$, where $\mbq$ is regarded as a $\sZ$-module by specializing $\up$ to $1$. The algebra $\afSrmbq$ is the affine Schur algebra over $\mbq$. For $A\in\afThnr$ the image of $[A]$ in $\afSrmbq$ will be denoted by $[A]_1$.

Let $\afmbzn:=\{(\la_i)_{i\in\mbz}\mid
\la_i\in\mbz,\,\la_i=\la_{i-n}\ \text{for}~\ i\in\mbz\}$, $\afmbnn:=\{(\la_i)_{i\in\mbz}\in\afmbzn\mid
\la_i\geq 0\ \}$. For $r\geq 0$, let
$\afLanr=\{\la\in\afmbnn\mid\sg(\la):=\sum_{1\leq i\leq n}\la_i=r\}.$
For $A\in\afThnpm$ and ${\bf j}\in \afmbnn$, define in
$\afSrmbq$ (cf. \cite[(3.0.3)]{Fu09})
$$A[{\bf j},r]=\sum_{\la\in\La_\vtg(n,r-\sg(A))}
\la^\bfj[A+\diag(\la)]_1,$$
where
$\la^\bfj=\prod_{i=1}^n\la_i^{j_i}$.

For $i\in I$ let  $\afbse_i\in\afmbnn$ be the element
satisfying $(\afbse_i)_j=\dt_{i,j}$ for $j\in I$.
By \cite[Thm. 6.1.5]{DDF}, there is an algebra homomorphism
\begin{equation}\label{surjective}
\eta_r:\sU(\afgl)\lra\afSrmbq,
\end{equation}
such that $\eta_r(\afE_{i,i})=0[\afbse_i,r]$ and $\eta_r(\afE_{i,j})=\afE_{i,j}[0,r]$ for $i\neq j$.

\section{Multiplication formulas in affine Schur algebras}

In \cite[Thm. 6.2.2]{DDF}, the multiplication formulas for $\afE_{i,j}[\bfl,r]A[\bfj,r]$ were derived in the case where either $|j-i|=1$ or $j=i+mn$ for some nonzero integer $m$. We will use it to derive the multiplication formulas for $\afE_{i,j}[\bfl,r]A[\bfj,r]$ for any $i\not=j$ in \ref{generalization of Multiplication Formulas at v=1}.

For $A\in\afThn$, let $\ro(A)=\big(\sum_{j\in\mathbb{Z}}a_{i,j}\big)_{i\in\mathbb{Z}}\ \ ~\text{and}~\ \ \co(A)=\big(\sum_{i\in\mathbb{Z}}a_{i,j}\big)_{j\in\mathbb{Z}}.$ Note that
if $A,B\in\afThnr$ is such that $\co(B)=\ro(A)$ then $[B]_1\cdot [A]_1=0$ in $\afSrmbq$.
For convenience, we set $[A]_1=0\in\afSrmbq$ if one
of the entries of $A$ is negative.
The following multiplication formulas in the affine Schur algebra
$\afSrmbq$ are proved in \cite[Prop. 6.2.3]{DDF}.
\begin{Thm}\label{MFforSBE}
Let $1\leq h\leq n$, $A=(a_{i,j})\in\afThnr$, and $\la=\ro(A)$. The following multiplication formulas hold in $\afSrmbq:$
\begin{itemize}
\item[(1)] If $\varepsilon\in\{1,-1\}$ and $\la_{h+\vep}\geq 1$, then
$$[E^\vtg_{h,h+\varepsilon}+
\diag(\la-\bse^\vtg_{h+\varepsilon})]_1\cdot[A]_1
=\sum_{{i\in\mbz}}
(a_{h,i}+1)[A+E_{h,i}^\vtg-E_{h+\varepsilon,i}^\vtg]_1.$$

\item[(2)] If $m\in\mbz\backslash\{0\}$ and $\la_h\geq 1$, then
$$[\afE_{h,h+mn}+\diag(\la-\afbse_h)]_1\cdot[A]_1=\sum_{s\in\mbz}(a_{h,s+mn}+1)[A+\afE_{h,s+mn}-\afE_{h,s}]_1.$$
\end{itemize}
\end{Thm}

Furthermore we have the following multiplication formulas in
$\afSrmbq$ given in \cite[Thm. 6.2.2]{DDF}. For simplicity, we  set $A[\bfj,r]=0$ if some off-diagonal entries of
$A$ are negative.
\begin{Thm}\label{Multiplication Formulas at v=1}
Assume $h,l\in\mbz$, $\bfj\in\afmbnn$, and $A\in\afThnpm$. The
following multiplication formulas hold in $\afSrmbq:$
\begin{itemize}
\item[(1)] $0[\afbse_l,r]  A[\bfj,r]=A[\bfj+\afbse_l,r]
+\bigl(\sum_{s\in\mbz}a_{l,s}\bigr)A[\bfj,r]$;
\item[(2)] for $\varepsilon\in\{1,-1\}$,
\begin{equation*}\label{MF1}
\begin{split}
 E^\vtg_{h,h+\varepsilon}[\bfl,r]  A[\bfj,r] =
 &\sum_{t\not=h,h+\varepsilon}
(a_{h,t}+1)(A+E^\vtg_{h,t}-E^\vtg_{h+\varepsilon,t})[\bfj,r] \\
&+\sum_{0\leq t\leq j_h}(-1)^t\bigg({j_h\atop t}\bigg)
(A-E^\vtg_{h+\varepsilon,h})[\bfj+(1-t)\afbse_h,r] \\
&+(a_{h,h+\varepsilon}+1)\sum_{0\leq t\leq j_{h+\varepsilon}}\bigg({j_{h+\varepsilon}\atop
t}\bigg)(A+E^\vtg_{h,h+\varepsilon})[\bfj-t\afbse_{h+\varepsilon},r];
\end{split}
\end{equation*}
\item[(3)] for $m\in \mbz\backslash\{0\}$,
\begin{equation*}
\begin{split}
\afE_{h,h+mn}[\bfl,r]A[\bfj,r] &=\sum_{s\not\in\{h,h-mn\}}(a_{h,s+mn}+1)(A+\afE_{h,s+mn}-\afE_{h,s})[\bfj,r]\\
&\quad+\sum_{0\leq t\leq j_h}(-1)^t\left({j_h\atop t}
\right)(A-\afE_{h,h-mn})[\bfj+(1-t)\afbse_h,r]\\
&\quad+(a_{h,h+mn}+1)\sum_{0\leq t\leq j_h}\left({j_h\atop
t}\right)(A+\afE_{h,h+mn})[\bfj-t\afbse_h,r].\\
\end{split}
\end{equation*}
\end{itemize}
\end{Thm}

We now use \ref{MFforSBE}(1) to prove the following multiplication formulas in $\afSrmbq$.
\begin{Lem}\label{generalization1 of MFforSBE}
Assume $1\leq i\not=j\leq n$. Let $A\in\afThnr$ and $\la=\ro(A)$ with $\la_j\geq 1$.
Then we have
$$[\afE_{i,j}+\diag(\la-\afbse_j)]_1\cdot[A]_1=\sum_{t\in\mbz}(a_{i,t}+1)[A+\afE_{i,t}
-\afE_{j,t}]_1.$$
\end{Lem}
\begin{proof}
We proceed by induction on $|j-i|$. The case $|j-i|=1$ follows from \ref{MFforSBE}(1). Now we assume $|j-i|>1$.
Let
$$
\vep_{i,j}=
\begin{cases}
-1 &\text{if $i<j$,}\\
1 &\text{if $i>j$.}
\end{cases}
$$
Then we have $|j+\vep_{i,j}-i|<|j-i|$ and
\begin{equation}\label{eq for generalization1 of MFforSBE}
\begin{split}
[\afE_{i,j}+\diag&(\la-\afbse_j)]_1=[\afE_{i,j+\vep_{i,j}}
+\diag(\la-\afbse_j)]_1\cdot[\afE_{j+\vep_{i,j},j}
+\diag(\la-\afbse_j))]_1\\
&-[\afE_{j+\vep_{i,j},j}
+\diag(\la-\afbse_j+\afbse_i-\afbse_{j+\vep_{i,j}})]_1
\cdot[\afE_{i,j+\vep_{i,j}}+\diag(\la-\afbse_{j+\vep_{i,j}})]_1.
\end{split}
\end{equation}
By the induction hypothesis we have
\begin{equation*}
\begin{split}
&\quad[\afE_{i,j+\vep_{i,j}}
+\diag(\la-\afbse_j)]_1\cdot[\afE_{j+\vep_{i,j},j}
+\diag(\la-\afbse_j))]_1\cdot[A]_1\\
&=\sum_{s,t\in\mbz}(a_{j+\vep_{i,j},t}+1)(a_{i,s}+1)
[A+\afE_{j+\vep_{i,j},t}-\afE_{j,t}+\afE_{i,s}
-\afE_{j+\vep_{i,j},s}]_1
\end{split}
\end{equation*}
and
\begin{equation*}
\begin{split}
&\quad
[\afE_{j+\vep_{i,j},j}
+\diag(\la-\afbse_j+\afbse_i-\afbse_{j+\vep_{i,j}})]_1
\cdot[\afE_{i,j+\vep_{i,j}}+\diag(\la-\afbse_{j+\vep_{i,j}})]_1\cdot
[A]_1\\
&=\sum_{s,t\in\mbz}(a_{j+\vep_{i,j},t}+1-\dt_{s,t})(a_{i,s}+1)
[A+\afE_{j+\vep_{i,j},t}-\afE_{j,t}+\afE_{i,s}
-\afE_{j+\vep_{i,j},s}]_1.
\end{split}
\end{equation*}
This together with \eqref{eq for generalization1 of MFforSBE} implies that
\begin{equation*}
\begin{split}
[\afE_{i,j}+\diag(\la-\afbse_j)]_1\cdot[A]_1
&=\sum_{s,t\in\mbz}\dt_{s,t}(a_{i,s}+1)
[A+\afE_{j+\vep_{i,j},t}-\afE_{j,t}+\afE_{i,s}
-\afE_{j+\vep_{i,j},s}]_1\\
&=\sum_{t\in\mbz}(a_{i,t}+1)
[A-\afE_{j,t}+\afE_{i,t}]_1
\end{split}
\end{equation*}
as required.
\end{proof}

Using \ref{MFforSBE}(2) and \ref{generalization1 of MFforSBE}, we can prove the following generalization of
\ref{MFforSBE}.
\begin{Prop}\label{generalization2 of MFforSBE}
Assume $1\leq i\leq n$, $j\in\mbz$ and $i\not=j$. Let
$A\in\afThnr$ and $\la=\ro(A)$ with $\la_j\geq 1$. Then
we have
$$[\afE_{i,j}+\diag(\la-\afbse_j)]_1\cdot[A]_1=\sum_{t\in\mbz}(a_{i,t}+1)[A+\afE_{i,t}
-\afE_{j,t}]_1.$$
\end{Prop}
\begin{proof}
We write $j=k+mn$, where $1\leq k\leq n$ and $m\in\mbz$. The case $k=i$ is given in \ref{MFforSBE}(2). We now assume $k\not=i$.
Then we have
\begin{equation}\label{eq for generalization2 of MFforSBE}
\begin{split}
[\afE_{i,j}+&\diag(\la-\afbse_j)]_1=[\afE_{i,k}+\diag(\la-\afbse_j)]_1
\cdot[\afE_{k,j}+\diag(\la-\afbse_j)]_1\\
&-
[\afE_{k,j}+\diag(\la-\afbse_j+\afbse_i-\afbse_k)]_1\cdot
[\afE_{i,k}+\diag(\la-\afbse_j)]_1.
\end{split}
\end{equation}
By \ref{MFforSBE}(2) and \ref{generalization1 of MFforSBE} we have
$$[\afE_{i,k}+\diag(\la-\afbse_j)]_1
\cdot[\afE_{k,j}+\diag(\la-\afbse_j)]_1\cdot[A]_1=
\sum_{s,t\in\mbz}(a_{k,t}+1)(a_{i,s}+1)[A+\afE_{k,t}-
\afE_{j,t}+\afE_{i,s}-\afE_{k,s}]_1$$
and
\begin{equation*}
\begin{split}
&\quad[\afE_{k,j}+\diag(\la-\afbse_j+\afbse_i-\afbse_k)]_1\cdot
[\afE_{i,k}+\diag(\la-\afbse_j)]_1\cdot[A]_1\\
&=
\sum_{s,t\in\mbz}(a_{k,t}-\dt_{s,t}+1)(a_{i,s}+1)[A+\afE_{k,t}-
\afE_{j,t}+\afE_{i,s}-\afE_{k,s}]_1
\end{split}
\end{equation*}
This together with \eqref{eq for generalization2 of MFforSBE} implies
that
\begin{equation*}
\begin{split}
[\afE_{i,j}+\diag(\la-\afbse_j)]_1\cdot[A]_1
&=\sum_{t\in\mbz}(a_{i,t}+1)[A-\afE_{j,t}+\afE_{i,t}]_1
\end{split}
\end{equation*}
as required.
\end{proof}

The following multiplication formulas, which can be
proved using \ref{generalization2 of MFforSBE} in a way similar to the proof of \cite[Thm. 6.2.2]{DDF}, are the generalization
of \ref{Multiplication Formulas at v=1}(2) and (3).
\begin{Prop}\label{generalization of Multiplication Formulas at v=1}
Assume $h\not=k\in\mbz$ and $A\in\afThnpm$. The
following multiplication formulas hold in $\afSrmbq$:
\begin{equation*}\label{MF1}
\begin{split}
 E^\vtg_{h,k}[\bfl,r]  A[\bfj,r] =
 &\sum_{\forall t\not=h,k}
(a_{h,t}+1)(A+E^\vtg_{h,t}-E^\vtg_{k,t})[\bfj,r] \\
&+\sum_{0\leq t\leq j_h}(-1)^t\bigg({j_h\atop t}\bigg)
(A-E^\vtg_{k,h})[\bfj+(1-t)\afbse_h,r] \\
&+(a_{h,k}+1)\sum_{0\leq t\leq j_{k}}
\bigg({j_{k}\atop
t}\bigg)(A+E^\vtg_{h,k})[\bfj-t\afbse_{k},r].
\end{split}
\end{equation*}
\end{Prop}

We end this section by constructing the PBW-basis for $\afSrmbq$.
Recall the notation $\sL^+,\sL^-$ introduced in \eqref{L+-}. Let $\sL=\sL^+\cup\sL^-$. For $A\in\afThnr$ let $\bssg(A)=(\sg_i(A))_{i\in\mbz}\in\afLanr$, where
\begin{equation*}\label{sgi(A)}
\sigma_i(A)=a_{i,i}+\sum\limits_{j<i}(a_{i,j}+a_{j,i}).
\end{equation*}
Using \ref{generalization of Multiplication Formulas at v=1}, one can prove the following triangular relation in affine Schur algebras.
\begin{Lem}\label{triangular formulas}
For $A\in\afThnpm$ and $\la\in\afLanr$ we have
\begin{equation*}\label{eqs1}
\begin{split}
&\prod_{(i,j)\in\sL^+}\frac{{\afE_{i,j}[0,r]}^{a_{i,j}}}{a_{i,j}!}
[\diag(\la)]_1\prod_{(i,j)\in\sL^-}\frac{{\afE_{i,j}[0,r]}^{a_{i,j}}}{a_{i,j}!}\\
&\quad=[A+\diag(\la-\bssg(A))]_1
+\sum_{B\in\afThnr\atop\sg(B^++B^-)<\sg(A)}f_{A,\la,B}[B]
\end{split}
\end{equation*}
where $f_{A,\la,B}\in\mbq$ and the products are taken with respect to any fixed total order on $\sL^+$ and $\sL^-$.
\end{Lem}

The following result is a direct consequence of \ref{triangular formulas}.

\begin{Coro}\label{basis for ASA}
The set
$$\bigg\{
\prod_{(i,j)\in\sL^+}{{\afE_{i,j}[0,r]}}^{a_{i,j}}[\diag(\bssg(A))]
\prod_{(i,j)\in\sL^-}{{\afE_{i,j}[0,r]}}^{a_{i,j}}
\mid A\in\afThnr
\bigg\}
$$
forms a basis for $\afSrl_\mbq$, where the products are taken with respect to any fixed total order on $\sL^+$ and $\sL^-$.
\end{Coro}

\section{The algebra $\bfTr$}

We now define an algebra $\bfTr$ as below. We will prove in \ref{main} that $\bfTr$ is isomorphic to the affine Schur algebra $\afSrmbq$.

\begin{Def}\label{T(n,r)} Let $n\geq 2$ and $I=\mathbb{Z}/n\mathbb{Z}$. Let $\bfTr$ be the associative algebra over $\mathbb{Q}$ generated by the elements
$${\bffke_i},{\bffkf}_i, \bffkk_\lambda,\tte_{i,i+mn}(i\in I, m\in\mbz\backslash\{0\}, \lambda\in\Lambda_\vtg(n,r))$$
subject to the relations:
\begin{itemize}
\item[(R1)]
$\bffkk_\lambda \bffkk_\mu = \delta_{\lambda,\mu}\bffkk_\lambda, \quad
\sum_{\lambda\in\Lambda_\vtg(n,r)}\bffkk_\lambda = 1$;
\item[(R2)]
$\bffke_i \bffkk_\lambda =\bffkk_{\lambda+\afal_i} \bffke_i$,
$\bffkf_i\bffkk_\lambda=\bffkk_{\lambda-\afal_i} \bffkf_i$ for $\la\in\afmbzn$,
where $\afal_i=\boldsymbol e_i^\vtg-\boldsymbol e_{i+1}^\vtg$ and $\bffkk_\la=0$ for $\la\not\in\afLanr$.
\item[(R3)]
$\bffke_i \bffkf_j - \bffkf_j \bffke_i = \delta_{i,j}\sum\limits_{\lambda\in\Lambda_\vtg(n,r)}
({\lambda_i}-\lambda_{i+1})\bffkk_\lambda$;
\item[(R4)] $\displaystyle\sum_{a+b=1-c_{i,j}}(-1)^a
\bpa{1-c_{i,j}}{a}  \bffke_i^{a}\bffke_j\bffke_i^{b}=0$ for
$i\not=j$;
\item[(R5)] $\displaystyle\sum_{a+b=1-c_{i,j}}(-1)^a
\bpa{1-c_{i,j}}{a} \bffkf_i^{a}\bffkf_j\bffkf_i^{b}=0$ for
$i\not=j$;
\item[(R6)]
$[\ttx_{i,m},[[\ldots[\tte_i,\tte_{i+1}],\ldots],\tte_n]]
=\tte_{1,1+mn}-\tte_{i,i+mn}$ for $i\neq 1$ and $m>0$ where $\ttx_{i,m}=[[\ldots,[[\tte_1,\tte_{2,2+(m-1)n}],\tte_2],\ldots],\tte_{i-1}]$;
\item[(R7)]
$[[[\ttf_n,\ldots,[\ttf_{i+1},\ttf_{i}]\ldots]],\tty_{i,m}]
=\tte_{1,1-mn}-\tte_{i,i-mn}$
for $i\neq 1$ and $m>0$, where $\tty_{i,m}=[\ttf_{i-1},[\ldots,[\ttf_2,
[\tte_{2,2-(m-1)n},\ttf_1]],\ldots]]$;
\item[(R8)]
$\tte_{i,i+mn}\bffkk_\lambda=\bffkk_\lambda\tte_{i,i+mn},$ $\tte_{i,i+mn}\tte_{j,j+ln}=\tte_{j,j+ln}\tte_{i,i+mn};$
\item[(R9)]$
\sum\limits_{1\leq i\leq n}\tte_{i,i+mn}\bffke_j=\bffke_j\sum\limits_{1\leq i\leq n}\tte_{i,i+mn},$ $\sum\limits_{1\leq i\leq n}\tte_{i,i+mn}\bffkf_j=\bffkf_j\sum\limits_{1\leq i\leq n}\tte_{i,i+mn};$
\item[(R10)]
$f_i({m_1},{m_2},\ldots,{m_{t}})\bffkk_{\la}=0$ for $i\in I$,
 $t\geq 1$, $m_1,\ldots,m_{t}\in\mathbb{Z}\backslash\{0\}$ and $\la_i<t.$
Here $f_i({m_1},{m_2},\ldots,{m_{t}})$ is defined recursively as follows. For $m\in\mbz\backslash\{0\}$, let $f_i(m)=\tte_{i,i+nm}$, and $f_i(0)=\sum_{\lambda\in\Lambda_\vtg(n,r)}\lambda_i\bffkk_\lambda$.
For $t> 1$ and $m_1,m_2,\ldots, m_t\in\mathbb{Z}$, let
\begin{equation}\label{recursive forumula}
f_i({m_1},\ldots,m_{t}) =f_i({m_1},\ldots,{m_{t-1}})f_i(m_{t})-\sum_{j=1}^{t-1}
f_i(m_1,\ldots,\widehat{m_j},\ldots,m_{t-1},m_j+m_{t}),
\end{equation}
where $\widehat{m_j}$ indicates that $m_j$ is omitted.
\end{itemize}
\end{Def}

The definition implies the following result.
\begin{Lem}\label{antilem0}
There is a unique $\mathbb{Q}$-algebra anti-automorphism $\tau$ on $\bfTr$
such that
$$\tau(\bffke_i) = \bffkf_i,~ \tau(\bffkf_i) = \bffke_i,~ \tau(\bffkk_\lambda) =\bffkk_\lambda,~~\tau(\tte_{i,i+mn}) =\tte_{i,i-mn},$$
for $i\in I$ and $m\in\mbz\backslash\{0\}$.
\end{Lem}

By \ref{lem0}, we conclude that there is an algebra homomorphism $\xi_r:\sU(\afgl)\rightarrow \bfTr$ such that
\begin{equation}\label{xir}
\xi_r(E_i)=\bffke_i,~~\xi_r(F_i)=\bffkf_i,~\xi_r(H_i)= \sum\limits_{\lambda\in\afLa(n,r)}\lambda_i\bffkk_\lambda~~\text{and}
~~\xi_r(\afE_{i,i+mn})\mapsto \tte_{i,i+mn}
\end{equation}
for $i\in I$, $m\in\mbz$ and $m\not=0$.
According to \cite[(4.1.1)]{Fu09}, we have the following result.
\begin{Lem}\label{surj xir}
For $\la\in\afLanr$ we have $$\xi_r\bigg(\bigg({H\atop\la}\bigg)\bigg)=
\begin{cases}
\ttk_\la&\text{if $\sg(\la)=r$;}
\\
0&\text{otherwise}.
\end{cases}$$
In particular the map $\xi_r$ is surjective.
\end{Lem}

Let \begin{equation*}
\tte_{i,j}:=\xi_r(\afE_{i,j})~i=1,\ldots,n,~j\in\mbz.
\end{equation*}
For $i\in\mbz$, let $\bar i$ denote the integer modulo $n$.
Since $[\afE_{i,j},\afE_{k,l}]=\dt_{\bar j,\bar
k}\afE_{i,l+j-k}-\dt_{\bar l,\bar i}\afE_{k,j+l-i}~\text{for}~i,j,k,l\in\mbz$,
we conclude that
\begin{equation}\label{eq2}
[\tte_{i,j},\tte_{k,l}]=\dt_{\bar j,\bar
k}\tte_{i,l+j-k}-\dt_{\bar l,\bar i}\tte_{k,j+l-i}.
\end{equation}
for $i,j,k,l\in\mbz$.

For $A\in\afThnp$,
let $$\bffke^{A}=\prod_{1\leq i\leq n\atop i<j,\,j\in\mbz}\tte_{i,j}^{a_{i,j}}=M_n\cdots M_2M_1,$$ where
$$
M_j=M_j(A)=
\prod_{1<j+sn\atop s\in\mbz}\tte_{1,j+sn}^{a_{1,j+sn}}
\prod_{2<j+sn\atop s\in\mbz}\tte_{2,j+sn}^{a_{2,j+sn}}\cdots\prod_{n<j+sn\atop s\in\mbz}
{\tte_{n,j+sn}^{a_{n,j+sn}}}.
$$
Note that by \eqref{eq2} we have $\tte_{i,j+mn}\tte_{i,j+ln}
=\tte_{i,j+ln}\tte_{i,j+mn}$ for $i,j\in I$ and $m,l\in\mbz$. So
the product $\prod_{i<j+sn,\,s\in\mbz}\tte_{i,j+sn}^{a_{i,j+sn}}$ is independent of the order of $s$.
For $A\in\afThnm$, let $\bffkf^{A}=\tau(\tte^{\tA})$, where $\tA$ is the transpose of $A$. Then we have
\begin{equation*}\bffkf^{A}=\prod_{1\leq i\leq n\atop i>j}\tte_{i,j}^{a_{i,j}}=M'_1M'_2\cdots M'_n,\end{equation*} where
\begin{align*}
M'_j=M_j'(A)=&\prod_{j+sn>n\atop s\in\mbz}\tte_{j+sn,n}^{a_{j+sn,n}}\prod_{j+sn>n-1\atop s\in\mbz}
\tte_{j+sn,n-1}^{a_{j+sn,n-1}}\cdots
\prod_{j+sn>1\atop s\in\mbz}\tte_{j+sn,1}^{a_{j+sn,1}}.
\end{align*}

For $A\in\afThnpm$ and $i,j\in\mbz$ let
\begin{equation}\label{Aij+}
A_{i,j}^+=\sum_{i<j+sn\atop s\in\mbz}a_{i,j+sn}\afE_{i,j+sn}\text{ and }
A_{j,i}^-=\sum_{i<j+sn\atop s\in\mbz}a_{j+sn,i}\afE_{j+sn,i}.
\end{equation}
Furthermore, let
\begin{equation}\label{Aj+}
A_j^+=\sum_{1\leq i\leq n}A_{i,j}^+=\sum_{i<j\atop i\in\mbz}a_{i,j}
\afE_{i,j}\text{ and }
A_j^-=\sum_{1\leq i\leq n}A_{j,i}^-=\sum_{i<j\atop i\in\mbz}a_{j,i}
\afE_{j,i}.
\end{equation}
Then we have $A^+=\sum_{1\leq j\leq n}A_j^+$,
$A^-=\sum_{1\leq j\leq n}A_j^-$.
Furthermore we have

$$
\tte^{A_{i,j}^+}=\prod_{i<j+sn\atop s\in\mbz}\tte_{i,j+sn}^{a_{i,j+sn}},\,
\ttf^{A_{j,i}^-}=\prod_{i<j+sn\atop s\in\mbz}\tte_{j+sn,i}^{a_{j+sn,i}},\,
\bffke^{A_j^+}= M_j(A^+),\
\bffkf^{A_j^-}= M'_j(A^-).$$
Hence
we have
\begin{equation}\label{main2}
\bffke^{A^+}=\bffke^{A_n^+}\cdots\bffke^{A_1^+},\,
\bffkf^{A^-}=\bffkf^{A_1^-}\cdots\bffkf^{A_n^-},\, \bffke^{A_j^+}=\bffke^{A_{1,j}^+}\cdots\bffke^{A_{n,j}^+},\,
\ttf^{A_j^-}=\ttf^{A_{j,n}^-}\cdots\ttf^{A_{j,1}^-}.
\end{equation}


For $k\in\mbn$ let
\begin{equation}\label{Nk}
\sN_k=\spann_\mbq\{
\tte^{A^+}\ttk_\mu\ttf^{A^-}\mid A\in\afThnpm,\,\sg(A)< k,\,\mu\in\afLanr\}.
\end{equation}
Furthermore, let
\begin{equation}\label{Nkpm}
\sN_k^+=\spann_\mbq\{
\tte^{A}\mid A\in\afThnp,\,\sg(A)< k\}\text{ and }
\sN_k^-=\spann_\mbq\{
\ttf^{A}\mid A\in\afThnm,\,\sg(A)< k\}.
\end{equation}
By \eqref{eq2} we have the following result.
\begin{Lem}\label{MM'}
Assume $M=\tte_{i_1,j_1}\cdots\tte_{i_k,j_k}$ and $M'=\tte_{i_1',j_1'}\cdots\tte_{i_l',j_l'}$, where $i_s\not=j_s$ and
$i_t'\not=j_t'$ for all $s,t$. Then:
\begin{itemize}
\item[(1)]
$MM'-M'M\in\sN_{k+l}.$
\item[(2)]
If $i_s<j_s$ and $i_t'<j_t'$ for all $s,t$ then $MM'-M'M\in\sN_{k+l}^+.$
\item[(3)]
If $i_s>j_s$ and $i_t'>j_t'$ for all $s,t$ then $MM'-M'M\in\sN_{k+l}^-.$
\end{itemize}
\end{Lem}

\begin{Lem}\label{lem4}
For $\la\in\afmbzn$ and $1\leq i\leq n$ and $j\in\mbz$,
we have  $\tte_{i,j}\ttk_\la=\ttk_{\la+\afbse_i-\afbse_j}\tte_{i,j}$, where $\ttk_\la=0$ for $\la\not\in\afLa(n,r)$.
\end{Lem}
\begin{proof}
We write $j=k+mn$ with $1\leq k\leq n$ and $m\in\mbz$. If $k=i$ then the assertion follows from \ref{T(n,r)}(R8). Now we assume $k\not=i$.
From \eqref{eq2} we see that $$\tte_{i,j}=\tte_{i,k+mn}=
\begin{cases}
[[\cdots[[\tte_{i,i+1},\tte_{i+1,i+2}],
\tte_{i+2,i+3}],\cdots,\tte_{k-1,k}],\tte_{k,k+mn}]&\text{if $i<k$;}
\\
[[\cdots[[\tte_{i,i-1},\tte_{i-1,i-2}],
\tte_{i-2,i-3}],\cdots,\tte_{k+1,k}],\tte_{k,k+mn}]&\text{if $i>k$}
\end{cases}$$
By \ref{T(n,r)}(R2) and (R8) we conclude that
$$\tte_{i,k+mn}\ttk_\la=
\begin{cases}
\ttk_{\la+\afal_i+\afal_{i+1}+\cdots+\afal_{k-1}}\tte_{i,k+mn}
=\ttk_{\la+\afbse_i-\afbse_k}\tte_{i,k+mn}&\text{if $i<k$;}
\\
\ttk_{\la-\afal_{i-1}-\afal_{i-2}-\cdots-\afal_{k}}\tte_{i,k+mn}
=\ttk_{\la+\afbse_i-\afbse_k}\tte_{i,k+mn}&\text{if $i>k$.}
\end{cases}$$
The assertion follows.
\end{proof}

\begin{Coro}\label{lem3}
Let $\la\in\afmbzn$.
\begin{itemize}
\item[(1)]
For $A\in\afThnp$ we have
$\bffke^{A}\bffkk_\lambda=\bffkk_{\la-\co(A)+\ro(A)}\bffke^{A}$.
\item[(2)]
For $A\in\afThnm$ we have
$\bffkk_\lambda\ttf^{A}=\ttf^{A}\bffkk_{\la+\co(A)-\ro(A)}.$
\end{itemize}
\end{Coro}
\begin{proof}
Applying $\tau$ defined in \ref{antilem0} to (1) gives
(2).  We prove (1). If $A\in\afThnp$, then by \ref{lem4} we have
$\bffke^{A}\bffkk_\lambda=\bffkk_{\la+\bt}\bffke^{A},$
where
$$\bt=\sum_{1\leq i\leq n\atop i<j,\,j\in\mbz}a_{i,j}(\afbse_i-\afbse_j)
=\sum_{1\leq i\leq n}\sum_{i<j\atop j\in\mbz}a_{i,j}\afbse_i-
\sum_{1\leq k\leq n}\sum_{1\leq i\leq n,\,s\in\mbz\atop i<k+sn}a_{i,k+sn}\afbse_{k}=\ro(A)-\co(A).$$
The assertion follows.
\end{proof}

For $1\leq i\leq n$, let $$\afThnpm_i=\{A\in\afThnpm\mid a_{h,t}=0\,\text{for}\,1\leq h\leq n,\,t\in\mbz\,\text{with}\,
(h,t)\not\in\{(i,i+kn)\mid k\in\mbz\}\}.$$
For $1\leq i\leq n$, let $\afThnp_i=\afThnp\cap\afThnpm_i$ and $\afThnm_i=\afThnm\cap\afThnpm_i$.
By definition
we have
\begin{equation}\label{spectial elements}
\tte^{A^+}=\prod_{m\geq 1}\tte_{i,i+mn}^{a_{i,i+mn}}=
\prod_{m\geq 1}f_i(m)^{a_{i,i+mn}}\text{ and }
\ttf^{A^-}=\prod_{m\geq 1}\tte_{i+mn,i}^{a_{i+mn,i}}
=\prod_{m\geq 1}f_i(-m)^{a_{i+mn,i}}
\end{equation}
for $A\in\afThnpm_i$. By \ref{T(n,r)}(R8) we have
\begin{equation}\label{commute}
\tte^{A^+}\ttk_\la=\ttk_\la\tte^{A^+},\
\ttf^{A^-}\ttk_\la=\ttk_\la\ttf^{A^-} \text{ and }
\tte^{A^+}\ttf^{A^-}=\ttf^{A^-}\tte^{A^+}
\end{equation}
for $A\in\afThnpm_i$ and $\la\in\afLanr$.

\begin{Lem}\label{fi(m1) cdots fi(mt)}
For $i\in I$ and $m_1,\cdots,m_t\in\mbz$ we have
$$f_i(m_1)\cdots f_i(m_t)=f_i(m_1,\cdots,m_t)+g$$
where $g$ is a $\mbz$-linear combination of $f_i(l_1)\cdots f_i(l_s)$
with $1\leq s\leq t-1$ and $l_1,\cdots,l_s\in\mbz$.
\end{Lem}
\begin{proof}
Let $\sX_{i,t}=\spann_{\mbz}\{f_i(l_1)\cdots f_i(l_s)\mid 1\leq s\leq t,\,l_1,\cdots,l_s\in\mbz\}$.
We proceed by induction on $t$. The case $t=1,2$ is trivial. Now we assume $t>2$.
By definition we have
$f_i({m_1},\ldots,m_{t}) =f_i({m_1},\ldots,{m_{t-1}})f_i(m_{t})-\sum_{j=1}^{t-1}
h_j,$ where $h_j=f_i(m_1,\ldots,\widehat{m_j},\ldots,m_{t-1},m_j+m_{t}).$ By the induction hypothesis, we have $f_i({m_1},\ldots,{m_{t-1}})f_i(m_{t})-f_i(m_1)\cdots f_i(m_{t-1})f_i(m_{t})\in\sX_{i,t-2}f_i(m_{t})\han \sX_{i,t-1}$ and $h_j\in\sX_{i,t-1}$ for all $j$. The assertion follows.
\end{proof}

\begin{Lem}\label{lem1}
Let $\la\in\Lambda_\vtg(n,r)$.  If $A\in\afThnpm_i$ and $\la_i<\sigma_i(A)$ for some $1\leq i\leq n$, then
\begin{equation}\label{cor11}
\bffke^{A^+}\bffkk_\lambda\bffkf^{A^-}=
\sum_{{B\in\afThnpm_i\atop \sigma(B)<\sigma(A)}}d_{B}\bffke^{B^+}\bffkk_\lambda\bffkf^{B^-},
\end{equation}
where $d_{B}\in\mathbb{Q}$.
\end{Lem}
\begin{proof}
Since $A\in\afThnpm_i$, there exist $m_1,m_2,\ldots,m_t\in\mbz\backslash\{0\}$ with $t=\sigma(A)=\sg_i(A)$ such that $\bffke^{A^+}\bffkf^{A^-}=f_i(m_1)\cdots f_i(m_t).$ It follows from \ref{fi(m1) cdots fi(mt)} that
$$\bffke^{A^+}\bffkf^{A^-}=f_i(m_1,\cdots,m_t)+g$$
where $g$ is a $\mbq$-linear combination of $f_i(l_1)\cdots f_i(l_s)$
with $1\leq s<t=\sg(A)$ and $l_1,\cdots,l_s\in\mbz$. Given $l_1,\cdots,l_s\in\mbz$ there exists $B\in\afThnpm_i$ and $k\in\mbn$ such that $f_i(l_1)\cdots f_i(l_s)=\tte^{B^+}\ttf^{B^-}f_i(0)^{k}$ and $s=\sg(B)+k$.
It follows that
\begin{equation}\label{eq for lem1}
\bffke^{A^+}\bffkf^{A^-}= f_i(m_1,\cdots,m_t)+\sum_{B\in\afThnpm_i\atop\sg(B)<\sg(A),\,k\in\mbn}
g_{B,k}\tte^{B^+}\ttf^{B^-}f_i(0)^{k}\quad(g_{B,k}\in\mbq)
\end{equation}
Since $\la_i<\sg_i(A)=t$, by \ref{T(n,r)}(R10) we have $f_i(m_1,\cdots,m_t)\ttk_\la=0$. Thus by \eqref{commute} and \eqref{eq for lem1}
we have
$$\bffke^{A^+}\ttk_\la\bffkf^{A^-}= \bffke^{A^+}\bffkf^{A^-}\ttk_\la= \sum_{B\in\afThnpm_i\atop\sg(B)<\sg(A),\,k\in\mbn}
g_{B,k}\la_i^{k}\tte^{B^+}\ttk_\la\ttf^{B^-}.$$
The assertion follows.
\end{proof}

Recall the notation $A_{i,j}^+$, $A_j^+$ defined in \eqref{Aij+} and \eqref{Aj+}.
For $1\leq i\leq n$ let
$$\Ga_i=\{A\in\afThnpm\mid A^+=A_i^+,\,A^-\in\afThnm_i\}\text{ and }
\Ga_i'=\{A\in\afThnpm\mid A^-=A_i^-,\,A^+\in\afThnp_i\}.$$
\begin{Coro}\label{cor1}
Let $\lambda\in\Lambda_\vtg(n,r)$ and $A\in\afThnpm$. If $A\in\Ga_i$ and $\la_i<\sigma_i(A)$ for some $1\leq i\leq n$, then
\begin{equation}\label{cor11}
\bffkf^{A^-}\bffke^{A^+}\bffkk_\lambda=
\sum_{C\in\Ga_i \atop
\sigma(C)<\sigma(A)}
g_{C}\bffkf^{C^-}\bffke^{C^+}\bffkk_\lambda\ (g_C\in\mbq).
\end{equation}
If $A\in\Ga_i'$ and $\la_i<\sigma_i(A)$, then
\begin{equation}\label{cor12}
\bffkk_\lambda\bffkf^{A^-}\bffke^{A^+}=
\sum_{C\in\Ga_i'\atop\sigma(C)<\sigma(A)}g_{C}'
\bffkk_\lambda\bffkf^{C^-}\bffke^{C^+}
\ (g_C'\in\mbq).
\end{equation}
\end{Coro}
\begin{proof}
Applying $\tau$ defined in \ref{antilem0} to \eqref{cor11} gives
\eqref{cor12}.
We prove \eqref{cor11}. Let $$\Pi_{i,A}=\spann_\mbq\{\ttf^{C^-}\tte^{C^+}
\mid C\in\Ga_i,\,\sg(C)<\sg(A)\}.$$
We have to show that $\bffkf^{A^-}\bffke^{A^+}\bffkk_\lambda\in\Pi_{i,A}\ttk_\la$.
By \eqref{eq2} and \eqref{main2} we have
\begin{equation*}\label{eq1 for cor11}
\tte^{A^+}=\tte^{A_i^+}=\tte^{A_{i,i}^+}
\tte^{A_{1,i}^+}\cdots \tte^{A_{i-1,i}^+}\tte^{A_{i+1,i}^+}
\cdots \tte^{A_{n,i}^+}+g
\end{equation*}
where $g$ is a $\mbq$-linear combination of $\tte^B$ with $B=B_i^+\in\afThnp$ and
$\sg(B)<\sg(A^+)$.
It follows that
$\bffkf^{A^-}\bffke^{A^+}\bffkk_\lambda
=\bffkf^{A^-}\bffke^{A_{i,i}^+}\bffke^{A_{1,i}^+}
\cdots\bffke^{A_{i-1,i}^+}\bffke^{A_{i+1,i}^+}\ldots\bffke^{A_{n,i}^+}
\ttk_\lambda+\ttf^{A^-}g\ttk_\la$.
Since
$\ttf^{A^-}g\in\spann_\mbq\{\ttf^{A^-}\tte^B\mid
B=B_i^+\in\afThnp,\,\sg(B)+\sg(A^-)<\sg(A)\}\han\Pi_{i,A},$
it is enough to prove that
\begin{equation}
\bffkf^{A^-}\bffke^{A_{i,i}^+}\bffke^{A_{1,i}^+}
\cdots\bffke^{A_{i-1,i}^+}\bffke^{A_{i+1,i}^+}\ldots\bffke^{A_{n,i}^+}
\ttk_\lambda\in\Pi_{i,A}\ttk_\la.
\end{equation}
By \ref{lem3} we have
\begin{equation}\label{eq2 for cor11}
\bffkf^{A^-}\bffke^{A_{i,i}^+}\bffke^{A_{1,i}^+}
\cdots\bffke^{A_{i-1,i}^+}\bffke^{A_{i+1,i}^+}\ldots\bffke^{A_{n,i}^+}
\ttk_\lambda=\bffkf^{A^-}\bffke^{A_{i,i}^+}\bffkk_{\lambda'}
\bffke^{A_{1,i}^+}\cdots\bffke^{A_{i-1,i}^+}
\bffke^{A_{i+1,i}^+}\ldots\bffke^{A_{n,i}^+}
\end{equation}
where $\la'=\la-\co(A^+-A_{i,i}^+)+\ro(A^+-A_{i,i}^+)$. Since $A^+=A_i^+$
we have $(\ro(A^+-A_{i,i}^+))_i=0$. This together with the condition $\la_i<\sg_i(A)$ implies that $\lambda_i'=\la_i-(\co(A^+-A_{i,i}^+))_i=\lambda_i-\sigma_i(A^+)+\sigma_i(A_{i,i}^+)
<\sigma_i(A^-)+\sigma_i(A_{i,i}^+)$. Hence by \ref{lem1} and \eqref{commute} we have
$$\ttf^{A^-}\tte^{A_{i,i}^+}\ttk_{\la'}=\sum_{B\in\afThnpm_i\atop\sg(B)<
\sg(A_{i,i}^+)+\sg(A^-)}d_B\ttf^{B^-}\tte^{B^+}\ttk_{\la'}$$
where $d_B\in\mbq$. This together with
\eqref{eq2 for cor11} implies that
\begin{equation*}
\begin{split}
&\qquad \bffkf^{A^-}\bffke^{A_{i,i}^+}\bffke^{A_{1,i}^+}
\cdots\bffke^{A_{i-1,i}^+}\bffke^{A_{i+1,i}^+}\ldots\bffke^{A_{n,i}^+}
\ttk_\lambda\\
&=\sum_{B\in\afThnpm_i\atop\sg(B)<
\sg(A_{i,i}^+)+\sg(A^-)}d_B\ttf^{B^-}\tte^{B^+}\ttk_{\la'}
\bffke^{A_{1,i}^+}
\cdots\bffke^{A_{i-1,i}^+}\bffke^{A_{i+1,i}^+}\ldots
\bffke^{A_{n,i}^+}\\
&=\sum_{B\in\afThnpm_i\atop\sg(B)<
\sg(A_{i,i}^+)+\sg(A^-)}d_B\ttf^{B^-}(\tte^{B^+}
\bffke^{A_{1,i}^+}
\cdots\bffke^{A_{i-1,i}^+}\bffke^{A_{i+1,i}^+}
\ldots\bffke^{A_{n,i}^+})\ttk_{\la}\\
\end{split}
\end{equation*}
By \eqref{eq2} we have
$$
\tte^{B^+}
\bffke^{A_{1,i}^+}
\cdots\bffke^{A_{i-1,i}^+}\bffke^{A_{i+1,i}^+}
\ldots\bffke^{A_{n,i}^+}
 =\sum_{C=C_i^+\in\afThnp\atop\sg(C)\leq\sg(B^+)+\sg(A^+-A_{i,i}^+)}
 k_{B^+,C}\tte^C\,(k_{B^+,C}\in\mbq).$$
Thus we have
$$\bffkf^{A^-}\bffke^{A_{i,i}^+}\bffke^{A_{1,i}^+}
\cdots\bffke^{A_{i-1,i}^+}\bffke^{A_{i+1,i}^+}\ldots\bffke^{A_{n,i}^+}
\ttk_\lambda=\sum_{B\in\afThnpm_i,\,C=C_i^+\in\afThnp\atop\sg(C)+\sg(B^-)
\leq\sg(B)+\sg(A^+-A_{i,i}^+)<
\sg(A)}d_Bk_{B^+,C}\ttf^{B^-}\tte^C\ttk_{\la}$$
as required.
\end{proof}

\begin{Prop}\label{prop1}
The set $$\mathfrak{M}:=\{\bffke^{A^+}\bffkk_{\la}\bffkf^{A^-}|A\in\afThnpm,\,
\la_i\geq\sg_i(A)\,\forall i\}$$ is a spanning set for $\bfTr$.
\end{Prop}
\begin{proof}
For $A\in\afThnpm$ and $\lambda\in\Lambda_\vtg(n,r)$, let
\begin{equation*}
\ttm_{A,\lambda}=\bffke^{A^+}\bffkk_\lambda\bffkf^{A^-}.
\end{equation*}
By \eqref{basis2} and \ref{surj xir}, we conclude that $\bfTr$ is spanned by the elements $\ttm_{A,\lambda}$ with $\lambda\in\afLa(n,r)$, $A\in\afThnpm$.
Therefore, to prove this proposition, we must show that if $\lambda\in\afLanr$, $B\in\afThnpm$ and $\lambda_i<\sigma_i(B)$ for some $1\leq i\leq n$, then
$\ttm_{B,\lambda}$ lies in the span of $\mathfrak{M}$.

We proceed by induction on $\sigma(B)$.
If $\sigma(B)=1$, then by \ref{T(n,r)}(R10) and \ref{lem4} we have
$\ttm_{B,\la}=0$. Assume now that $\sigma(B)>1$ and $\lambda_i<\sigma_i(B)$.
By \eqref{main2} we may write
$$\ttm_{B,\lambda}=\bffke^{B_n^+}\cdots\bffke^{B_1^+}\bffkk_\lambda\bffkf^{B_1^-}\cdots\bffkf^{B_n^-},$$
where $B_i^+$ and $B_i^-$ are as given in \eqref{Aj+}.
Let ${\ttm}_{B,\lambda}^{(1)}=
x_B\bffke^{B_i^+}
\bffkk_\lambda\bffkf^{B_i^-}y_B.$
where $x_B=\bffke^{B_n^+}\cdots\bffke^{B_{i+1}^+}\bffke^{B_{i-1}^+}
\cdots\bffke^{B_1^+}$ and $y_B=\bffkf^{B_1^-}
\cdots\bffkf^{B_{i-1}^-}\bffkf^{B_{i+1}^-}\cdots\bffkf^{B_n^-}$.
By \ref{MM'} we have
\begin{equation}\label{eq1 for prop1}
\ttm_{B,\la}-\ttm_{B,\la}^{(1)} \in\sN_{\sg(B)}.
\end{equation}
Furthermore by  \ref{lem3} we have
$${\ttm}_{B,\lambda}^{(1)}=
x_B\bffkk_{\lambda'}\bffke^{B_i^+}\bffkf^{B_i^-}
y_B,$$
where $\lambda'=\la-\co(B_i^+)+\ro(B_i^+).$
Let ${\ttm}_{B,\lambda}^{(2)}=
x_B\bffkk_{\lambda'}\bffkf^{B_i^-}\bffke^{B_i^+}
y_B.$
Then by \ref{MM'} we have
\begin{equation}\label{eq2 for prop1}
\ttm_{B,\la}^{(1)}-\ttm_{B,\la}^{(2)} \in\sN_{\sg(B)}.
\end{equation}
Let $${\ttm}_{B,\lambda}^{(3)}=
x_B(\bffkk_{\lambda'}\bffkf^{B_i^-}\bffke^{B_{i,i}^+})\bffke^{B_{1,i}^+}
\cdots\bffke^{B_{i-1,i}^+}\bffke^{B_{i+1,i}^+}
\cdots\bffke^{B_{n,i}^+}
y_B.$$
By \eqref{main2} and \ref{MM'} we have
$$\tte^{B_i^+}-\bffke^{B_{i,i}^+}\bffke^{B_{1,i}^+}
\cdots\bffke^{B_{i-1,i}^+}\bffke^{B_{i+1,i}^+}
\cdots\bffke^{B_{n,i}^+}\in\sN_{\sg(B_i^+)}^+.$$
where
$B_{s,i}^+$ is as given in \eqref{Aij+}.
It follows that
\begin{equation}\label{eq3 for prop1}
\ttm_{B,\la}^{(2)}-\ttm_{B,\la}^{(3)} \in\sN_{\sg(B)}.
\end{equation}
Note that $B_{i,i}^+\in\afThnp_i.$
Since $\lambda_i<\sigma_i(B)=\sigma_i(B_i^-)+\sigma_i(B_i^+)$,  we have
$\lambda'_i=\la_i-\sg_i(B_i^+)+\sg_i(B_{i,i}^+)
<\sigma_i(B_i^-)+\sigma_i(B_{i,i}^+)$.
Hence by \eqref{cor12} and \ref{MM'} we have
$$\bffkk_{\lambda'}\bffkf^{B_i^-}\bffke^{B_{i,i}^+}\in\sN_{\sg(B_i^-+B_{i,i}^+)}.
$$
This together with \ref{MM'} implies that
\begin{equation}\label{eq4 for prop1}
{\ttm}_{B,\lambda}^{(3)}\in\sN_{\sg(B)}.
\end{equation}
By the induction hypothesis we have $\sN_{\sg(B)}\han\spann\frak{M}$.
Thus by \eqref{eq1 for prop1}--\eqref{eq4 for prop1} we have
$\ttm_{B,\la}\in\spann\frak{M}$. The assertion follows.
\end{proof}

\section{The isomorphism between $\bfTr$ and $\afSrmbq$}

For $m\in\mbz$ let
$$
\tf_i(m)=
\begin{cases}
\afE_{i,i+nm}[0,r]&\text{if $m\not=0$;}\\
0[\afbse_i,r]&\text{otherwise}.
\end{cases}$$
Furthermore, by using \eqref{recursive forumula},
we define similarly the elements $\tf_i(m_1,\cdots,m_t)\in\afSrmbq$  for $m_1,\cdots,m_t\in\mbz$ with $t>1$.

\begin{Lem}\label{key1}
 For $i\in I$ and $m_1,\ldots,m_t\in\mbz\backslash\{0\}$, we have
$$\tf_i({m_1},{m_2},\ldots,{m_{t}})=a_{m_1,\ldots,m_t}\cdot
\bigg(\sum_{j=1}^{t}\afE_{i,i+m_{j}n}\bigg)[0,r].$$\ where $a_{m_1,\ldots,m_t}= \prod_{2\leq k\leq t}(\sum_{1\leq s\leq k}\dt_{m_s,m_k})$.
\end{Lem}
\begin{proof}
We use induction on $t$.
The case $t=1$ is trivial. The result follows from \ref{Multiplication Formulas at v=1} when $t=2$.
Assume now $t>2$.
By the inductive hypothesis
we have $\tf_i({m_1},{m_2},\ldots,{m_{t-1}})=a_{m_1,\ldots,m_{t-1}}\cdot A[0,r]$, where
$A=\sum_{s=1}^{t-1}\afE_{i,i+m_{s}n}$. It follows from \ref{Multiplication Formulas at v=1}(3) that
\begin{align*}
&\quad \tf_i({m_1},{m_2},\ldots,{m_{t}})\\
&=\tf_i(m_{t})\tf_i({m_1},{m_2},\ldots,{m_{t-1}})-\sum_{j=1}^{t-1}
\tf_i(m_1,\ldots,\widehat{m_j},\ldots,m_{t-1},m_j+m_{t})
\\
&=a_{m_1,\ldots,m_{t-1}}\sum_{j=1}^{t-1}
(h_j(A+\afE_{i,i+n(m_j+m_{t})}-\afE_{i,i+nm_j})[0,r]+
s_j(A-\afE_{i,i-nm_{t}})[\afbse_i,r])\\&
\quad
+a_{m_1,\ldots,m_t}\big(\sum_{j=1}^{t}\afE_{i,i+m_{j}n}\big)[0,r]-\sum_{j=1}^{t-1}
\tf_i(m_1,\ldots,\widehat{m_j},\ldots,m_{t-1},m_j+m_{t}),
\end{align*}
where $h_j=(1+\sum_{1\leq s\leq t-1,\,s\not=j}\dt_{m_j+m_t,m_s})
(1-\delta_{0,m_j+m_{t}})(\sum_{1\leq s\leq t-1}\dt_{m_j,m_s})^{-1}$
and $s_j=\delta_{0,m_j+m_{t}}(\sum_{1\leq s\leq t-1}\dt_{m_j,m_s})^{-1}$.
Thus it is enough to prove
\begin{equation*}\label{eq0 for key1}
\begin{split}
&\qquad\tf_i(m_1,\ldots,\widehat{m_j},\ldots,m_{t-1},m_j+m_{t})\\
&=
a_{m_1,\ldots,m_{t-1}}
(h_j(A+\afE_{i,i+n(m_j+m_{t})}-\afE_{i,i+nm_j})[0,r]+
s_j(A-\afE_{i,i-nm_{t}})[\afbse_i,r])
\end{split}
\end{equation*}
for $1\leq j\leq t-1$.

If $m_j+m_{t}=0$ for some $1\leq j\leq t-1$, then by definition
we have
\begin{equation}\label{eq1 for key1}
\begin{split}
&\quad\quad\tf_i(m_1,\ldots,\widehat{m_j},\ldots,m_{t-1},m_j+m_{t})\\
&=\tf_i({m_1},\ldots,\widehat{m_j},\ldots,m_{t-1})\tf_i(0)
-\sum\limits_{1\leq k\leq t-1\atop k\neq j}
\tf_i(m_1,\ldots,\widehat{m_j},\ldots,
\widehat{m_k},\ldots,m_{t-1},m_k).
\end{split}
\end{equation}
Furthermore we have $h_j=0$ and $s_j=(\sum_{1\leq s\leq t-1}\dt_{m_j,m_s})^{-1}$.
Since $a_{m_1,\ldots,m_k,m_{k+1},\ldots,m_{t-1}}=a_{m_1,\ldots,m_{k+1},m_k,\ldots,m_{t-1}}$ for any $k$ we have
\begin{equation}\label{am1 cdots mt-1}
a_{m_1,\ldots,m_{t-1}}=a_{m_1,\ldots,\widehat {m_j},\ldots,m_{t-1},m_j}
=a_{m_1,\ldots,\widehat {m_j},\ldots,m_{t-1}}\sum_{1\leq s\leq t-1}\dt_{m_j,m_s}=a_{m_1,\ldots,\widehat {m_j},\ldots,m_{t-1}}s_j^{-1}.
\end{equation}
Thus by the induction hypothesis we conclude that
\begin{equation*}
\begin{split}
\tf_i(m_1,\ldots,\widehat{m_j},\ldots,
\widehat{m_k},\ldots,m_{t-1},m_k)&=\tf_i({m_1},\ldots,\widehat{m_j},
\ldots,m_{t-1})\\
&={a_{m_1,\ldots,\h{m_j},\ldots,m_{t-1}}}
(A-\afE_{i,i+nm_j})[0,r]\\
&= {a_{m_1,\ldots,m_{t-1}}}s_j
(A-\afE_{i,i+nm_j})[0,r]
\end{split}
\end{equation*}
for $1\leq k\leq t-1$ with $k\not=j$.
This together with \eqref{eq1 for key1} and \ref{Multiplication Formulas at v=1}(1) implies that
\begin{equation*}
\begin{split}
\tf_i(m_1,\ldots,\widehat{m_j},\ldots,m_{t-1},m_j+m_{t})&=
\tf_i(m_1,\ldots,\h{m_j},\ldots,m_{t-1})(\tf_i(0)-(t-2))\\
&= {a_{m_1,\ldots,m_{t-1}}}s_j
(A-\afE_{i,i+nm_j})[0,r](0[\afbse_i,r]-(t-2))\\
&= {a_{m_1,\ldots,m_{t-1}}}s_j
(A-\afE_{i+inm_j})[\afbse_i,r]
\end{split}
\end{equation*}
as desired.

Now we assume $m_j+m_{t}\neq 0$ for some $1\leq j\leq t-1$. Then we have $s_j=0$
and $h_j=(1+\sum_{1\leq s\leq t-1,\,s\not=j}\dt_{m_j+m_t,m_s})
(\sum_{1\leq s\leq t-1}\dt_{m_j,m_s})^{-1}$.
By the induction hypothesis and
\eqref{am1 cdots mt-1} we have
\begin{equation*}
\begin{split}
&\qquad\tf_i(m_1,\ldots,\widehat{m_j},\ldots,m_{t-1},m_j+m_{t})\\
&=a_{m_1,\ldots,\widehat{m_j},
\ldots,m_{t-1},m_j+m_{t}}(A+\afE_{i,i+n(m_j+m_{t})}-\afE_{i,i+nm_j})[0,r]\\
&=a_{m_1,\ldots,\widehat{m_j},
\ldots,m_{t-1}}(1+\sum_{1\leq s\leq t-1,\,s\not=j}\dt_{m_j+m_t,m_s})(A+\afE_{i,i+n(m_j+m_{t})}-\afE_{i,i+nm_j})[0,r]\\
&=a_{m_1,\ldots,m_{t-1}}h_j(A+\afE_{i,i+n(m_j+m_{t})}-\afE_{i,i+nm_j})[0,r].
\end{split}
\end{equation*}
The assertion follows.
\end{proof}

\begin{Thm}\label{main}
The map $\eta_r$ given in \eqref{surjective} induces an algebra isomorphism
$\bar\eta_r: \bfTr\ra\afSrl_\mbq$ such that $\bar\eta_r\circ\xi_r=\eta_r$, where $\xi_r$ is defined in \eqref{xir}.
In particular, $\afSrl_\mbq$ is generated by $${\bffke_i},{\bffkf}_i, \bffkk_\la,\tte_{i,i+nm}(1\leq i\leq n, m\in\mbz\backslash\{0\}, \la\in\Lambda_\vtg(n,r))$$ subject to the relations {\rm (R1)--(R10)} in $\ref{T(n,r)}$.
\end{Thm}
\begin{proof}
By \ref{lem0} and \ref{key1}, the map $\eta_r$ induces a surjective algebra homomorphism
\begin{equation*}
\bar\eta_r: \bfTr\ra\afSrl_\mbq
\end{equation*}
such that $\bar\eta_r\circ\xi_r=\eta_r$.
By \ref{basis for ASA} and \ref{prop1} we see that the map $\bar\eta_r$ sends
a spanning set of $\bfTr$ to a basis of $\afSrmbq$. Hence $\bar\eta_r$ is an isomorphism.
\end{proof}

\end{document}